\documentclass{amsart}

\usepackage{mathrsfs, mathtools, amssymb}

\usepackage{dsfont}

\usepackage[paper=a4paper, margin=3cm]{geometry}

\usepackage[breaklinks]{hyperref}       
\usepackage{url}            
\usepackage{booktabs}       
\usepackage{amsfonts}       
\usepackage{nicefrac}       
\usepackage{microtype}      
\usepackage{lipsum}

\usepackage{amsmath,amssymb,amsthm,amscd}
\usepackage{amsrefs}
\usepackage{tikz,tikz-cd}
\usetikzlibrary{arrows,arrows.meta,calc}

\usepackage{tikz-3dplot}

\usepackage{graphicx}
\graphicspath{}
\DeclareGraphicsExtensions{.pdf,.png,.jpg,.jpeg}

\usepackage{yhmath}
\usepackage{enumitem}
\usepackage{breakurl}
\usepackage{subcaption}

\usepackage{tikz}
\usepackage{tikz-cd}
\usetikzlibrary{arrows}
\usepackage{epsfig}
\usepackage[all]{xy}
\usepackage{epstopdf}
\usepackage{framed}
\usepackage[dvipsnames]{xcolor}

\newcommand{\Z}{\mathbb Z}
\newcommand{\Q}{\mathbb Q}
\newcommand{\R}{\mathbb R}
\newcommand{\C}{\mathbb C}

\newcommand{\lb}{\lbrace}
\newcommand{\rb}{\rbrace}
\newcommand{\la}{\langle}
\newcommand{\ra}{\rangle}
\renewcommand{\phi}{\varphi}
\newcommand{\eps}{\varepsilon}

\DeclareMathOperator{\rk}{rk}

\DeclareMathOperator{\cone}{cone}

\DeclareMathOperator{\Hom}{Hom}

\DeclareMathOperator{\Id}{Id}

\DeclareMathOperator{\Img}{Im}
\DeclareMathOperator{\GL}{GL}

\DeclareMathOperator{\Homs}{\mathscr{H}\text{\kern -3pt {\calligra\large om}}\,}

\DeclareMathOperator{\str}{star}

\DeclareMathOperator{\lk}{lk}

\DeclareMathOperator{\depth}{depth}

\DeclareMathOperator{\rint}{int}
\DeclareMathOperator{\const}{const}
\DeclareMathOperator{\SL}{SL}
\DeclareMathOperator{\PSL}{PSL}
\DeclareMathOperator{\cat}{cat}
\DeclareMathOperator{\Inc}{Inc}

\renewcommand{\leq}{\leqslant}
\renewcommand{\geq}{\geqslant}

\theoremstyle{plain}
\newtheorem{thm}{Theorem}
\newtheorem{lm}[thm]{Lemma}

\theoremstyle{definition}

\newtheorem{ex}[thm]{Example}
\newtheorem{defn}[thm]{Definition}

\tikzcdset{arrow style=tikz, diagrams={>=stealth}}

\begin{document}
	
	\title[Brieskorn spheres and torus manifolds]{
		Brieskorn spheres bound orbit spaces of torus manifolds
	}
	
	\author{Grigory Solomadin}
	\address[G.\,Solomadin]{Philipps-Universit\"at Marburg, Germany}
	\email{grigory.solomadin@gmail.com}
	
	\begin{abstract}
		We construct a smooth, equivariantly formal, simply connected torus $8$-manifold over a contractible $4$-manifold with the boundary a Brieskorn homology $3$-sphere having a nontrivial fundamental group. 
		We prove extension to a torus graph for GKM graphs of GKM$_{4}$ manifolds in complexity $2$, and $3$ under a finite fundamental group assumption (that is not satisfied, in general).
	\end{abstract}
	
	
		
	\date{\today}
	\maketitle

	Consider the Brieskorn manifold
	\[
		\Sigma(p,q,r):=\lb z=(z_{0},z_{1},z_{2})\in \C^{3}\colon z_{0}^{p}+z_{1}^{q}+z_{2}^{r}=0,\ 
		|z|=1\rb.
	\]
	For pairwise coprime integers $p,q,r>1$ it is a homology $3$-sphere having a nontrivial (perfect) fundamental group 
	\[
		\pi_{1}(\Sigma(p,q,r))\cong \la u,v,w\mid u^{p}=v^{q}=w^{r}=uvw \ra.
	\]
	In this note, by definition a \emph{torus manifold} $M^{2n}$ is a closed, simply connected and \emph{equivariantly formal}, i.e. $H^{odd}(M^{2n};\Z)=0$ holds~\cite{ma-pa-06}, smooth or topological manifold equipped with an effective action of the torus $T^{n}=(S^{1})^{n}$ having a non-empty and finite fixed point set.
	It is known that the fundamental group $\pi_{1}(Q_{i})$ is perfect for a smooth torus manifold $M^{2n}$ and $2\leq i\leq n$~\cite[Thm.~2]{ay-ma-23}, where $Q_{i}$ denotes the $i$-skeleton of the orbit space $Q=M^{2n}/T^{n}$.
	In this paper, we construct torus manifolds with a non-simply connected $2$-skeleton of the orbit space.

	\begin{thm}\label{thm:top}
		For any pairwise coprime integers $p,q,r>1$, there exists a topological torus manifold $M^{8}$ with the contractible orbit space $Q=M^{8}/T^{4}$ with boundary $Q_{3}=\Sigma(p,q,r)$ having $\pi_{1}(Q_{2})=\pi_{1}(Q_{3})$ a nontrivial perfect group.
	\end{thm}

	The proof of theorem~\ref{thm:top} is based on the Davis-Januszkiewicz construction applied to a manifold bounding $\Sigma(p,q,r)$ with the topological face structure dual to a double triangulation on the boundary.
	(The homology sphere property is necessary for the proof of equivariant formality of $M^{8}$.)
	For the relevant statement in the setting of smooth torus actions, one has to consider a smaller class of orbit spaces.
	For example, as it is well known, the Poincar\'e sphere $\Sigma(2,3,5)$ does not bound a smooth contractible manifold (e.g. see~\cite[\S2.3]{ma-16}).
	On the other hand, the Brieskorn sphere $\Sigma(2,5,7)$ bounds a smooth contractible Mazur manifold~\cite{ak-ki-79}.
	The proof strategy does not carry over to the smooth setting directly, since the double-triangulation face structure on $\partial Q$ is not smooth, in general.
	Namely, there are PL-creases at the boundaries of the respective cells (outside the respective corners), e.g. see figure~\ref{fig:cell}. 
	We are able to prove the smooth version of the above theorem, using a different method.

	\begin{thm}\label{thm:mainsm}
		For any pairwise coprime integers $p,q,r>1$ such that $\Sigma(p,q,r)$ bounds a smooth contractible manifold $W^{4}_{m}$, there exists a locally standard, smooth torus manifold $N^{8}$ with the contractible orbit space $Q=N^{8}/T^{4}$ with boundary $Q_{3}$ that is homotopy equivalent to $\Sigma(p,q,r)$ having $\pi_{1}(Q_{2})=\pi_{1}(Q_{3})$ a nontrivial perfect group.		
	\end{thm}

	The proof of theorem~\ref{thm:mainsm} is based on the construction of~\cite{ku-ka-25}, which requires a smooth manifold with corners structure on the collar of $\Sigma(p,q,r)$.
	We produce such a structure (not changing the homotopy type of the boundary) by first constructing a smooth conical stratification of $\Sigma(p,q,r)\times [0,1)$, then applying the real oriented blow-up of this conical stratification.
	Equivariant formality of the resulting space $N^{8}$ follows from face acyclicity of the orbit space $Q$~\cite{ma-pa-06}, because all faces in the manifold with corners $Q$ are contractible.
	Furthermore, the resulting faces are homeomorphic to the standard disks in dimensions $\leq 3$ by the Poincar\'e conjecture solution.
	This shows that the face structure on $Q_{3}$ is that of a regular CW complex structure, implying $\pi_{1}(Q_{2})=\pi_{1}(Q_{3})$.
	The provided conical stratification holds more generally for all Brieskorn manifolds.
	However, by Poincar\'e--Lefschetz duality, only homology spheres among them bound contractible manifolds, and orbit spaces of torus manifolds are contractible (by face acyclicity and lifting of curves from $Q$ to $M^{2n}$, combined with the Whitehead theorem).
	Since $\partial Q$ is a regular CW complex by the construction, theorems~\ref{thm:top},~\ref{thm:mainsm} imply existence of a (topological or smooth, respectively) torus manifold $M^{2n}$ with a nontrivial $\pi_{1}(Q_{2})=\pi_{1}(Q_{3})$ for all $n\geq 4$ by taking Cartesian product with $\C P^{n-4}$ equipped with the standard $T^{n-4}$-action.

	The regular CW complex $Q_{2}$ for any torus manifold $N^{8}$ given by theorem~\ref{thm:mainsm} does not admit a \emph{shelling}~\cite{bj-84}, which follows by contradiction from inductively applying the Seifert--Van Kampen theorem.
	Admitting a shelling was of importance in the study of a particular class of torus $T^{k}$ actions on $M^{2n}$ of arbitrary \emph{complexity} $c=n-k$ called \emph{GKM manifolds}~\cite{go-so-25}, and served as the original motivation for the search of examples constructed in theorem~\ref{thm:mainsm}.
	A GKM manifold~\cite{gu-za-01} admits a labelling on the graph $Q_{1}$; such a decorated graph satisfies several additional properties and is called its \emph{GKM graph}; in the case $n=k$ this graph is called a \emph{torus graph} (for the relevant definitions of GKM theory see~\S\ref{sec:thm3proof} below).
	The property $\pi_{1}(Q_{2})=1$ for a GKM$_{4}$ manifold implies extension of its GKM graph to a torus graph~\cite{ay-ma-so-23}.
	In the case of a nontrivial $\pi_{1}(Q_{2})$ of a GKM$_{4}$ manifold, one can prove extension of GKM graphs to torus graphs in complexity $1$~\cite[Thm.~4]{go-so-25}.
	We provide a further enhancement of~\cite[Thm.~4]{go-so-25},~\cite[Thm.~2]{go-so-26}, leading to a generators-and-relations description of the respective graph equivariant cohomology ring as follows (the necessary ring-theoretical definitions can be found e.g. in~\cite{go-so-26}).
	
	\begin{thm}\label{thm:mainext}
		Let $M^{2n}$ be a GKM$_{4}$ manifold with the torus $T^{k}$-action, $\rk T^{k}=k$.
		If $n-k=2$, or $n-k=3$ and $\pi_{1}(Q_{2})$ is finite, then $\Gamma$ extends to a torus graph.
		In both cases, the graph equivariant cohomology ring $H_{T}^{*}(\Gamma;\Bbbk)$ is isomorphic to the quotient of the face ring $\Bbbk[S(\Gamma)]$ by the ideal generated by some $n-k$ elements of degree $2$.
		Here, either $\Bbbk=\Q$, or $\Bbbk=\Z$ and all stabilizers of the GKM manifold $M^{2n}$ are connected.
	\end{thm}

	The proof of the first part for theorem~\ref{thm:mainext} follows by a short additional to~\cite[Thm.~4]{go-so-25} argument involving Minkovski's theorem on the subgroups in $SL_{c}(\Z)$, $c=2,3$.
	The second part then follows trivially from the results of~\cite[Thm.~2]{go-so-26}.
	
	\section{A topological torus manifold}\label{sec:top}

\begin{proof}[Proof of theorem~\ref{thm:top}]
	By the result of Freedman~\cite[Thm.~1.4']{fr-82}, any integer homology $3$-sphere bounds a contractible topological manifold.
	Let $W^{4}$ be such a topological manifold with $\partial W^{4}=\Sigma(p,q,r)$.
	Choose any triangulation $K_{1}$ of $\Sigma(p,q,r)$ (which exists by~\cite[Thm.~10.6]{mu-61}) and let $K$ be the barycentric subdivision of $K_{1}$.
	The dual cell
	\[
		D(I):=
		\lb \sigma=(I_{0}\subset\cdots\subset I_{j})\in K'\colon I\subseteq I_{0}\rb\subseteq
		|K'|=|K|,
	\] 
	of any simplex $I\in K$, $I\neq\varnothing$, in the barycentric subdivision $K'$ of $K$ is homeomorphic to $\cone(\lk_{K} I)$, and the link 
	\[
		\lk_{K} I:=
		\lb J\in K\colon I\sqcup J\in K \rb
	\]
	of $I\neq\varnothing$ has dimension $2-\dim I\leq 2$ and therefore is a standard sphere up to a homeomorphism.
	Additionally define $D(\varnothing):=W^{4}$.
	The cells $D(I)$, $I\in K$, together with the maximal cell, constitute a topological nice manifold with corners structure on $W^{4}$.
	Namely, any point of $\rint D(I)$ has a neighborhood homeomorphic to $\R^{4-|I|}\times \R^{|I|}_{\geq 0}$, and the facets $D(v)$, $v\in I$, map to the respective coordinate hypersurfaces (see figure~\ref{fig:cell}). 
	Denote such a manifold with corners by $K^{*}$.
	
	Let $v_{\sigma}$ be the vertex of $K$ corresponding to the simplex $\sigma\in K_{1}$.
	There is the canonical coloring of vertices in $K$ by $[3]:=\lb 0,1,2,3\rb$ given by $v_{\sigma}\mapsto \dim \sigma$.
	It is proper, i.e. the vertices of any $i$-dimensional simplex in $K$ are colored in $i+1$ pairwise distinct colors for all $i\leq 3$.
	This coloring induces the characteristic function $\lambda\colon \mathcal{F}(K^{*})\to \Z^{4}$ on the set $\mathcal{F}(K^{*})$ of facets for $K^{*}$ to the simplicial complex $K$ (a so-called \emph{pullback from a linear model}~\cite{da-ja-91}).
	\begin{figure}
		\centering
		\input{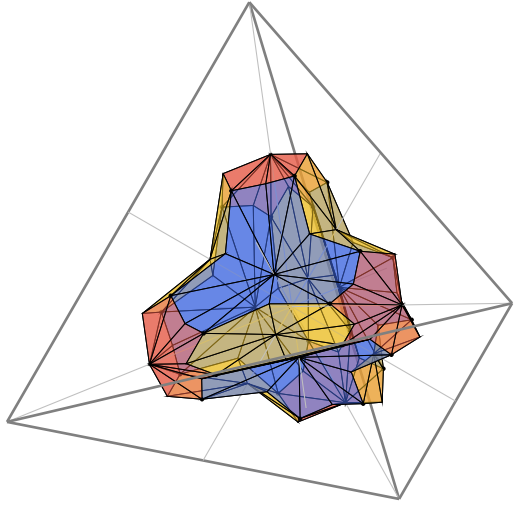}
		\caption{The cell of a tetrahedron (formed in its second barycentric subdivision) dual to its barycenter is a $PL$-model for the $3$-dimensional permutohedron\label{fig:cell}}		
	\end{figure}	
	The Davis-Januszkiewicz construction~\cite{da-ja-91} (for $T=T^{4}:=(S^{1})^{4}$) gives the topological torus manifold
	\[
		M^{8}=M(W^{4},\lambda):=
		\bigl(W^{4}\times T^{4}\bigr)/\sim,\ 
		(x,t)\sim (x',t')\Leftrightarrow 
		\bigl(x=x'\in \rint D(I)\ \&\ t^{-1}t'\in T(I)\bigr).
	\]
	Here, $T(I)\subseteq T^{4}$ is the subgroup spanned by the images of homomorphisms $\lambda(F)\in \Z^{4}\cong\Hom(S^{1},T^{4})$ in $T^{4}$, where $F$ runs over all facets $F\supseteq D(I)$ in $K^{*}$.
	(We put $T(\varnothing):=\lb 1\rb$.)
	The topological manifold property (and furthermore, local standardness of the $T$-action) is proved in~\cite{da-ja-91} (also see~\cite{ko-ku-25} for the detailed modern treatment for the topological properties of the Davis-Januszkiewicz construction over a stratified base).
	The Borel space of the $T$-action on $M^{8}$ is homotopy equivalent to the \emph{Davis-Januszkiewicz space}
	\[
		DJ(K):=
		\bigcup_{I\in K} \prod_{i\in I} \C P^{\infty}\times \prod_{j\in [N]\setminus I} \lb*\rb \subset (\C 	P^{\infty})^{N},
	\]
 	for the simplicial complex $K$ on the vertex set $[N]$, where $\lb *\rb$ denotes a based point.
	The $T^{4}$-space $M^8=M(W^{4},\lambda)$ equals
	\[
		\biggl(\bigsqcup_{I\in \cat K} D(I)\times \bigl(T/T(I)\bigr)\biggr)/\sim',\ 
		(x,\Inc_{I\to J}J)\sim' (p_{I\to J}x,J),
	\]
 	for projections $p_{I\to J}\colon T/T(I)\to T/T(J)$ of tori, $I\subset J$, where $\cat K$ denotes the small category of all simplices (and $\varnothing$ as the initial object) in $K$ with inclusions as morphisms, and $\Inc_{I\to J}\colon D(J)\to D(I)$ is induced by the natural inclusion of geometric realizations of simplices.
	(This is nothing but the homotopy colimit of the respective $(\cat K)$-diagram of tori by taking the union of simplices in $K'$, i.e. chains of simplices in $K$, with the least fixed element in the chain.)
	The Borel space of this $T^{4}$-space is the homotopy colimit of the composition for the classifying functor $B$ with the $T$-diagram of coordinate tori $T(I)$ (stabilizers of the torus action) in $T^{4}$, e.g. see~\cite[\S4.1]{df-04}.
	The groups $T(I)$ are just coordinate subgroups in $T^4$, and the respective diagram maps are their embeddings.
	Then the homotopy colimit amounts to taking the colimit (i.e. the respective diagram is Reedy cofibrant), which is the union of their classifying subspaces in $(\C P^{\infty})^{N}$, that is, $DJ(K)$.
	For the alternative argument see~\cite{da-ja-91}.

	We claim that the algebra $\Bbbk[K]$ is Cohen-Macaulay over fields $\Bbbk=\Q,\Z/p\Z$ for any prime $p$.
	Namely, the vanishing 
	\[
		\widetilde{H}_{i}(\lk_{K} I;\Z)=0,\ i<2-\dim I,\ I\in K, 
	\]
	of the reduced integral homology for the links implies the corresponding homology vanishing with coefficients $\Bbbk=\Q,\Z/p\Z$ for any prime $p$ by the Universal Coefficients Theorem.
	Then one applies Reisner's theorem, e.g. see~\cite[Thm.~3.3.10]{bu-pa-15}.
	This homology vanishing condition holds, since the links $\lk_{K} I$ for $\varnothing\neq I\in K$ are spheres by the above~\cite{mo-51}, and $\lk_{K} \varnothing=K$ is a homology sphere $\Sigma(p,q,r)$ by the construction.
	
	The Stanley-Reisner ring $\Z[K]\cong H^{*}(DJ(K);\Z)\cong H^{*}_{T}(M^{8};\Z)$ has the $H^{*}(BT)$-module structure given by the coloring of vertices for $K$.
	The image $\theta_{i}\in \Z[K]$, $i=1,\dots,4$, of a basis for of $H^{2}(BT;\Z)$ is a l.s.o.p. in $\Z[K]$ by unimodularity of $\lambda$ and~\cite[Lm.~3.3.2]{bu-pa-15}.
	Let $R^{*}:=H^{*}(BT;\Z)$ be a graded ring.
	The ring $\Bbbk[K]$ is free over $R^{*}\otimes\Bbbk$ for $\Bbbk=\Z/p\Z,\Q$ with $p$ any prime.
	There is the Koszul resolution $K^{*}(\underline{\theta},\Z)$ of the $R^{*}$-module $\Z$ defined by the l.s.o.p. $\underline{\theta}:=(\theta_{1},\dots,\theta_{4})$.
	One has $K^{*}(\underline{\theta},\Z)\otimes_{\Z}\Bbbk=K^{*}(\underline{\theta},\Bbbk)$, and such a complex of $R^{*}\otimes\Bbbk$-modules is exact for any $\Bbbk=\Q,\Z/p\Z$ by the above.
	By Cohen-Macaulayness, the complex $\Bbbk[K]\otimes_{R^{*}} K^{*}(\underline{\theta},\Z)$ is exact for any $\Bbbk=\Q,\Z/p\Z$.
	Hence, the complex $\Z[K]\otimes_{R^{*}} K^{*}(\underline{\theta},\Z)$ is exact by the UCT.
	This implies that $\theta_{i}$ is a regular sequence for $\Z[K]$ as well.
	Therefore, $\Z[K]/R^{>0}\Z[K]$ is torsion-free, i.e. a free $\Z$-module.
	Then by the graded Nakayama lemma $\Z[K]$ is a free $R^{*}$-module.
	We conclude that $M^{8}$ is equivariantly formal. 
	
	To find $\pi_{1}(M^{8})$ let $U_{f}:=\pi^{-1}(\rint W^{4})$ be the free part of the $T^{4}$-action on $M^{8}$, where $\pi\colon M^{8}\to W^{4}$ is the projection onto the orbit space, and let $U_{c}$ be the open collar neighbourhood of $\pi^{-1}(\partial W^{4})$ (induced by that of the boundary of the topological manifold $W^{4}$).
	Any principal $T^{4}$-bundle over $\pi(U_{f}\cap U_{c})\simeq \Sigma(p,q,r)$ (where $\simeq$ denotes a homotopy equivalence, which holds by the collar theorem for topological manifolds with boundary) is trivial, since the respective classifying map has a trivial homotopy class in
	\[
		[\Sigma(p,q,r),BT^{4}]\cong H^{2}(\Sigma(p,q,r);\Z^{4})=0.
	\]
	This applies in particular to the fiber bundle $T^{4}\to U_{f}\cap U_{c}\to  \Sigma(p,q,r)$.
	The map 
	\[
		\pi_{1}(U_{f}\cap U_{c})\cong 	\pi_{1}(U_{c})\times \Z^{4}\to \Z^{4}\cong \pi_{1}(U_{f})
	\] 
	induced by the inclusion is the projection onto the second factor.
	This follows directly from the homotopy long exact sequences for the trivial $T^{4}$-bundle embedding $U_{f}\cap U_{c}\to U_{f}$.
	The map
	\[
		\pi_{1}(U_{f}\cap U_{c})\cong 	\pi_{1}(U_{c})\times \Z^{4}\to \pi_{1}(U_{c})
	\] 
	induced by another inclusion is the projection onto the first factor, since there is a collection of $S^{1}$-generators of $T^{4}$ that are fibers of the projection and are collapsed each to a point (by unimodularity of $\lambda$).
	Hence, $\pi_{1}(M^{8})=1$ follows by the Seifert--Van Kampen theorem.
	
	Recall that $Q_{i}$ is the union in $Q$ of $\leq i$-dimensional closed faces $D(I)$, $I\in K$, $4-|I|\leq i$, which are $i$-dimensional topological disks, or $Q$ for $i=4$.
	In particular, $Q_{3}=\partial W^{4}$ is a regular CW complex obtained by gluing topological disks to $Q_{2}$.
	Hence, 
	\[
		\pi_{1}(Q_{2})\cong
		\pi_{1}(Q_{3})\cong	
		\pi_{1}(\Sigma(p,q,r)).
	\]
	The proof is complete.
	\end{proof}
	
	\section{A smooth torus manifold}\label{sec:sm}

	\subsection{Real oriented blow-ups and conical stratifications}

	We refer to~\cite{ga-04} for the definition of a real, oriented, spherical blow-up (\textit{real blow-up}, for short).
	The construction given below corresponds to the case when the respective building set is chosen to be maximal in the respective conical stratification.

	An $n$-dimensional (smooth) \textit{manifold with corners} (e.g. see~\cite{jo-12}) is a paracompact Hausdorff topological space $X$ equipped with an atlas $\lb(U_{a},\phi_{a})\rb_{a\in A}$ of \textit{charts with corners}, where $\phi_{a}\colon U_{a}\subseteq \R^{j}_{\geq 0}\times \R^{n-j}\to X$ is a homeomorphism from an open subset $U_{a}$ for some $0\leq j\leq n$, and the transition maps $\phi_{b}^{-1}\circ\phi_{a}$ are invertible and smooth (i.e. extend to smooth maps from the relatively open subset to an open neighborhood in the respective Euclidean space $\R^{j}\times \R^{n-j}$) for all $a,b\in A$.
	There is the well-defined function $\depth_{X}\colon X\to \Z_{\geq 0}$ sending $x\in X$ to the number of vanishing coordinates in a chart with corners.
	The boundary $\partial X$ of $X$ is defined as in~\cite{jo-12}; notice that neither $\partial X$ is connected, nor $\iota\colon\partial X\to X$ is an embedding, in general.
	A \emph{boundary hypersurface} of a manifold with corners $X$ is a closure of a connected component in $\iota(\partial X)\subset X$.

	A \textit{spherical slice} is the intersection of $\R^{a}\times \R^{b}_{\geq 0}$ with the open unit ball $B_{<1}(0)$ centered at the origin.
	A \textit{local real blow-up} of $D=D_{T}\times D_{N}$, where $D_{T}$ is a disk and $D_{N}$ is a spherical slice, is the map
	\[
		p\colon B_{0} D:=
		D_{T}\times [0,1)\times S_{N}\to D_{T}\times D_{N}=D,\ (r,t,l)\mapsto (r,tl),
	\]
	Here, $S_{N}$ is the intersection of the spherical slice closure with a sphere of unit radius centered at the origin.
	
	Let $M$ be a smooth manifold with corners, equipped with a smooth atlas with corners.
	Call a closed subset $X\subseteq M$ \textit{admissible} if at every $x\in X$ there is an open subspace $D(X,x)\ni x$ of $M$ that is diffeomorphic to the Cartesian product $D_{T}\times D_{N}$ of a disk $D_{T}$ (or its sector) and a spherical slice $D_{N}$, such that $X\cap D(X,x)$ maps to $D_{T}\times \lb 0\rb$.
	Then by definition the \textit{real blow-up} $p\colon B_{X} M\to M$ of $M$ along an admissible $X$ is the map defined on the manifold with corners $B_{X} M$ obtained by gluing together the local real blow-ups, defined as above.
	Let $E_{X}:=p^{-1}(X)\subset B_{X} M$ be the \textit{exceptional divisor} of the blow-up.

	One can iterate real blow-ups by ``the balls, beams and plates'' construction, starting from a collection $\mathcal{C}$ of locally closed submanifolds in a smooth manifold $M$ with corners by blowing-up the strict transforms of the closures for strata in any order that does not decrease the respective strata dimensions.
	For the real blow-up steps to be well-defined (namely, for the admissibility of the blow-up centers to hold), it is sufficient to assume that $\mathcal{C}$ is a smooth conical stratification, as defined below.
	See figure~\ref{fig:stratblowup} for the relevant example.

	\begin{figure}[ht]\centering
		\begin{tikzpicture}[scale=3,line join=round,line cap=round]
\definecolor{cT}{RGB}{65,105,225}
\definecolor{cV}{RGB}{230,90,70}
\definecolor{cE}{RGB}{235,190,40}
\definecolor{arc}{RGB}{42,42,42}
\fill[cT,fill opacity=0.20] (-0.816,-0.471) -- (-0.758,-0.533) -- (-0.687,-0.593) -- (-0.615,-0.644) -- (-0.536,-0.689) -- (-0.448,-0.730) -- (-0.360,-0.761) -- (-0.263,-0.787) -- (-0.170,-0.804) -- (-0.069,-0.814) -- (0.027,-0.816) -- (0.129,-0.809) -- (0.223,-0.796) -- (0.321,-0.773) -- (0.410,-0.744) -- (0.495,-0.709) -- (0.576,-0.667) -- (0.642,-0.626) -- (0.708,-0.577) -- (0.764,-0.527) -- (0.816,-0.471) -- (0.843,-0.378) -- (0.859,-0.287) -- (0.866,-0.198) -- (0.865,-0.114) -- (0.856,-0.022) -- (0.840,0.069) -- (0.816,0.160) -- (0.786,0.244) -- (0.751,0.325) -- (0.706,0.410) -- (0.659,0.486) -- (0.602,0.562) -- (0.540,0.634) -- (0.477,0.695) -- (0.406,0.755) -- (0.336,0.804) -- (0.258,0.850) -- (0.178,0.889) -- (0.085,0.922) -- (-0.003,0.942) -- (-0.074,0.925) -- (-0.146,0.902) -- (-0.221,0.869) -- (-0.290,0.833) -- (-0.367,0.784) -- (-0.445,0.724) -- (-0.514,0.661) -- (-0.578,0.591) -- (-0.641,0.512) -- (-0.693,0.432) -- (-0.743,0.343) -- (-0.782,0.256) -- (-0.814,0.166) -- (-0.840,0.069) -- (-0.856,-0.022) -- (-0.865,-0.120) -- (-0.865,-0.210) -- (-0.857,-0.298) -- (-0.840,-0.389) -- (-0.816,-0.471) -- cycle;
\fill[cT,fill opacity=0.20] (-0.866,-0.500) -- (-0.816,-0.471) -- (-0.778,-0.513) -- (-0.732,-0.556) -- (-0.635,-0.631) -- (-0.526,-0.694) -- (-0.407,-0.745) -- (-0.312,-0.775) -- (-0.207,-0.798) -- (-0.107,-0.811) -- (0.002,-0.816) -- (0.110,-0.811) -- (0.211,-0.798) -- (0.315,-0.775) -- (0.410,-0.744) -- (0.528,-0.693) -- (0.632,-0.632) -- (0.730,-0.558) -- (0.817,-0.472) -- (0.866,-0.500) -- (0.798,-0.603) -- (0.754,-0.657) -- (0.712,-0.702) -- (0.668,-0.744) -- (0.615,-0.789) -- (0.565,-0.825) -- (0.506,-0.862) -- (0.453,-0.892) -- (0.389,-0.921) -- (0.332,-0.943) -- (0.265,-0.964) -- (0.198,-0.980) -- (0.137,-0.991) -- (0.068,-0.998) -- (0.007,-1.000) -- (-0.063,-0.998) -- (-0.124,-0.992) -- (-0.193,-0.981) -- (-0.253,-0.968) -- (-0.312,-0.950) -- (-0.377,-0.926) -- (-0.433,-0.901) -- (-0.495,-0.869) -- (-0.547,-0.837) -- (-0.604,-0.797) -- (-0.652,-0.758) -- (-0.703,-0.711) -- (-0.745,-0.667) -- (-0.790,-0.613) -- (-0.826,-0.563) -- (-0.863,-0.504) -- cycle;
\fill[cT,fill opacity=0.20] (0.866,-0.500) -- (0.816,-0.471) -- (0.833,-0.417) -- (0.848,-0.356) -- (0.864,-0.234) -- (0.864,-0.108) -- (0.849,0.020) -- (0.826,0.124) -- (0.795,0.220) -- (0.756,0.314) -- (0.706,0.410) -- (0.648,0.501) -- (0.586,0.582) -- (0.518,0.656) -- (0.440,0.728) -- (0.336,0.804) -- (0.232,0.864) -- (0.118,0.911) -- (-0.000,0.944) -- (-0.000,1.000) -- (0.068,0.998) -- (0.129,0.992) -- (0.198,0.980) -- (0.257,0.966) -- (0.324,0.946) -- (0.381,0.925) -- (0.445,0.896) -- (0.499,0.867) -- (0.558,0.830) -- (0.608,0.794) -- (0.655,0.755) -- (0.706,0.708) -- (0.754,0.657) -- (0.793,0.609) -- (0.833,0.553) -- (0.866,0.501) -- (0.895,0.447) -- (0.924,0.383) -- (0.964,0.268) -- (0.978,0.208) -- (0.990,0.139) -- (0.998,0.070) -- (1.000,0.009) -- (0.999,-0.052) -- (0.993,-0.122) -- (0.982,-0.191) -- (0.968,-0.251) -- (0.948,-0.318) -- (0.927,-0.375) -- (0.902,-0.431) -- (0.870,-0.493) -- cycle;
\fill[cT,fill opacity=0.20] (-0.866,-0.500) -- (-0.816,-0.471) -- (-0.833,-0.417) -- (-0.848,-0.356) -- (-0.864,-0.234) -- (-0.864,-0.108) -- (-0.849,0.020) -- (-0.828,0.118) -- (-0.795,0.220) -- (-0.756,0.314) -- (-0.706,0.410) -- (-0.648,0.501) -- (-0.586,0.582) -- (-0.518,0.656) -- (-0.440,0.728) -- (-0.336,0.804) -- (-0.232,0.864) -- (-0.118,0.911) -- (-0.000,0.944) -- (-0.000,1.000) -- (-0.124,0.992) -- (-0.193,0.981) -- (-0.253,0.968) -- (-0.312,0.950) -- (-0.377,0.926) -- (-0.441,0.898) -- (-0.495,0.869) -- (-0.547,0.837) -- (-0.604,0.797) -- (-0.652,0.758) -- (-0.703,0.711) -- (-0.745,0.667) -- (-0.790,0.613) -- (-0.831,0.556) -- (-0.863,0.504) -- (-0.897,0.443) -- (-0.922,0.387) -- (-0.947,0.322) -- (-0.965,0.263) -- (-0.981,0.195) -- (-0.991,0.135) -- (-0.997,0.074) -- (-1.000,0.004) -- (-0.998,-0.065) -- (-0.992,-0.126) -- (-0.982,-0.187) -- (-0.967,-0.255) -- (-0.950,-0.314) -- (-0.925,-0.379) -- (-0.900,-0.435) -- (-0.868,-0.497) -- cycle;
\fill[cT,fill opacity=0.95] (0.000,-0.001) -- (-0.866,-0.500) -- (-0.831,-0.556) -- (-0.795,-0.606) -- (-0.757,-0.654) -- (-0.709,-0.705) -- (-0.665,-0.747) -- (-0.611,-0.791) -- (-0.562,-0.827) -- (-0.503,-0.865) -- (-0.417,-0.909) -- (-0.328,-0.945) -- (-0.227,-0.974) -- (-0.133,-0.991) -- (-0.037,-0.999) -- (0.068,-0.998) -- (0.163,-0.987) -- (0.257,-0.966) -- (0.357,-0.934) -- (0.445,-0.896) -- (0.529,-0.849) -- (0.608,-0.794) -- (0.681,-0.732) -- (0.754,-0.657) -- (0.814,-0.581) -- (0.866,-0.501) -- (0.000,-0.001) -- cycle;
\fill[cT,fill opacity=0.95] (0.000,-0.001) -- (0.866,-0.500) -- (0.899,-0.439) -- (0.924,-0.383) -- (0.948,-0.318) -- (0.966,-0.259) -- (0.982,-0.191) -- (0.991,-0.131) -- (1.000,-0.009) -- (0.996,0.087) -- (0.983,0.182) -- (0.961,0.276) -- (0.927,0.375) -- (0.883,0.470) -- (0.833,0.553) -- (0.777,0.630) -- (0.712,0.702) -- (0.642,0.767) -- (0.565,0.825) -- (0.484,0.875) -- (0.389,0.921) -- (0.299,0.954) -- (0.198,0.980) -- (0.103,0.995) -- (0.007,1.000) -- (-0.000,1.000) -- (0.000,-0.001) -- cycle;
\fill[cT,fill opacity=0.95] (0.000,-0.001) -- (-0.866,-0.500) -- (-0.922,-0.387) -- (-0.944,-0.330) -- (-0.965,-0.263) -- (-0.979,-0.204) -- (-0.991,-0.135) -- (-0.998,-0.065) -- (-1.000,-0.004) -- (-0.995,0.100) -- (-0.981,0.195) -- (-0.957,0.289) -- (-0.925,0.379) -- (-0.885,0.466) -- (-0.836,0.549) -- (-0.779,0.627) -- (-0.709,0.705) -- (-0.639,0.770) -- (-0.554,0.832) -- (-0.472,0.882) -- (-0.385,0.923) -- (-0.295,0.956) -- (-0.193,0.981) -- (-0.098,0.995) -- (-0.002,1.000) -- (-0.000,1.000) -- (0.000,-0.001) -- cycle;
\draw[arc,line width=1.00pt] (1.000,0.000) -- (0.996,0.087) -- (0.982,0.191) -- (0.961,0.276) -- (0.927,0.375) -- (0.883,0.470) -- (0.829,0.560) -- (0.765,0.644) -- (0.706,0.708) -- (0.642,0.767) -- (0.558,0.830) -- (0.484,0.875) -- (0.389,0.921) -- (0.291,0.957) -- (0.206,0.979) -- (0.120,0.993) -- (0.015,1.000) -- (-0.072,0.997) -- (-0.176,0.984) -- (-0.278,0.961) -- (-0.361,0.933) -- (-0.457,0.890) -- (-0.547,0.837) -- (-0.618,0.786) -- (-0.697,0.717) -- (-0.768,0.640) -- (-0.821,0.571) -- (-0.876,0.482) -- (-0.915,0.403) -- (-0.952,0.305) -- (-0.979,0.204) -- (-0.995,0.100) -- (-1.000,0.013) -- (-0.996,-0.092) -- (-0.984,-0.178) -- (-0.960,-0.280) -- (-0.932,-0.363) -- (-0.889,-0.458) -- (-0.836,-0.549) -- (-0.774,-0.633) -- (-0.716,-0.699) -- (-0.652,-0.758) -- (-0.569,-0.822) -- (-0.495,-0.869) -- (-0.401,-0.916) -- (-0.303,-0.953) -- (-0.219,-0.976) -- (-0.116,-0.993) -- (-0.011,-1.000) -- (0.094,-0.996) -- (0.180,-0.984) -- (0.265,-0.964) -- (0.365,-0.931) -- (0.460,-0.888) -- (0.551,-0.835) -- (0.635,-0.772) -- (0.700,-0.714) -- (0.760,-0.650) -- (0.824,-0.567) -- (0.878,-0.478) -- (0.917,-0.399) -- (0.948,-0.318) -- (0.976,-0.217) -- (0.994,-0.113) -- (1.000,-0.009) -- cycle;
\draw[cE,line width=0.40pt] (-0.866,-0.500) -- (-0.817,-0.472);
\draw[cE,line width=1.00pt] (0.000,-0.001) -- (-0.866,-0.500);
\draw[cE,line width=0.40pt] (0.866,-0.500) -- (0.817,-0.472);
\draw[cE,line width=1.00pt] (0.000,-0.001) -- (0.866,-0.500);
\draw[cE,line width=0.40pt] (-0.000,1.000) -- (-0.000,0.944);
\draw[cE,line width=1.00pt] (0.000,-0.001) -- (-0.000,1.000);
\draw[cE,line width=0.40pt] (-0.816,-0.471) -- (-0.745,-0.544) -- (-0.659,-0.614) -- (-0.568,-0.672) -- (-0.465,-0.723) -- (-0.354,-0.763) -- (-0.239,-0.793) -- (-0.126,-0.810) -- (-0.005,-0.816) -- (0.110,-0.811) -- (0.229,-0.794) -- (0.345,-0.766) -- (0.456,-0.726) -- (0.555,-0.679) -- (0.652,-0.619) -- (0.739,-0.550) -- (0.815,-0.473);
\draw[cE,line width=0.40pt] (-0.816,-0.471) -- (-0.844,-0.373) -- (-0.861,-0.263) -- (-0.866,-0.156) -- (-0.858,-0.041) -- (-0.838,0.075) -- (-0.806,0.190) -- (-0.764,0.297) -- (-0.709,0.404) -- (-0.648,0.501) -- (-0.574,0.596) -- (-0.491,0.682) -- (-0.401,0.759) -- (-0.311,0.821) -- (-0.211,0.874) -- (-0.107,0.915) -- (-0.003,0.942);
\draw[cE,line width=0.40pt] (0.816,-0.471) -- (0.845,-0.367) -- (0.861,-0.263) -- (0.866,-0.156) -- (0.858,-0.041) -- (0.838,0.075) -- (0.808,0.184) -- (0.764,0.297) -- (0.709,0.404) -- (0.648,0.501) -- (0.574,0.596) -- (0.491,0.682) -- (0.401,0.759) -- (0.311,0.821) -- (0.211,0.874) -- (0.107,0.915) -- (0.003,0.942);
\fill[cV,fill opacity=1.00] (0.000,-0.001) circle (0.030);
\fill[cV,fill opacity=0.20] (-0.816,-0.471) circle (0.030);
\fill[cV,fill opacity=0.20] (0.816,-0.471) circle (0.030);
\fill[cV,fill opacity=0.20] (-0.000,0.943) circle (0.030);
\end{tikzpicture}
		\input{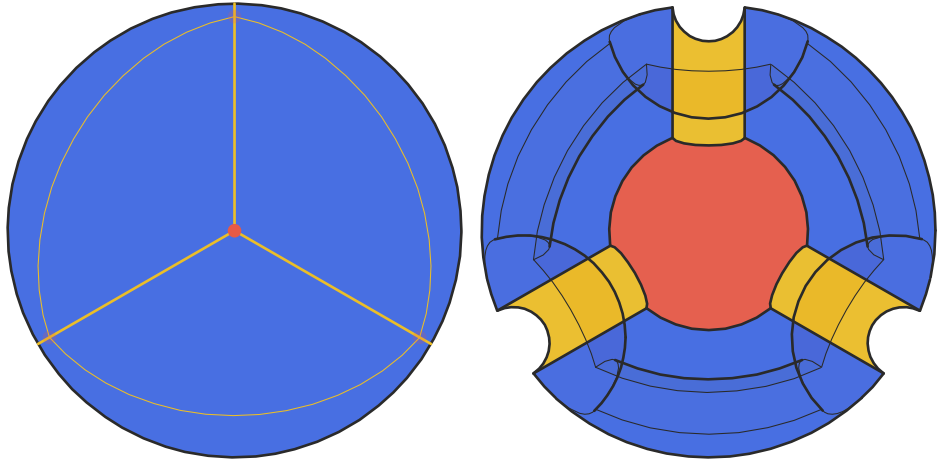}
		\caption{The conical stratification of $D^{3}$ by a geodesic tetrahedron on the boundary (left). Its blow-up (right) is a smooth manifold with corners modelling the $3$-dimensional permutohedron, compare with figure~\ref{fig:cell}}\label{fig:stratblowup}
	\end{figure}
	
	\begin{defn}[{\cite[\S7]{ga-04}}]\label{def:resdec}
		Let $M$ be a smooth $n$-dimensional manifold with corners, and let $\mathcal{C}$ be a finite collection of locally closed, connected, smooth submanifolds (without boundary) in $M$ with $M=\bigsqcup\limits_{X\in \mathcal{C}} X$.
		The collection $\mathcal{C}$ is called a \emph{(smooth) conical stratification} of $M$ if either $n=0$, or the following hold:
		\begin{enumerate}
			\item[(C1)] There is a unique stratum of maximal dimension $n$ in $\mathcal{C}$; for each $X\in\mathcal{C}$, $\overline{X}\setminus X$ is a union of elements in $\mathcal{C}$; every boundary hypersurface of $M$ is a union of elements in $\mathcal{C}$; this equips $\mathcal{C}$ with the partial order $X<Y\Leftrightarrow X\subset \overline{Y}$;
			
			\item[(C2)] Every $\overline{X}\in\mathcal{C}$ is admissible.
			Let $l(X):=S_{N}\subseteq M$ (using intersection with a sphere of small radius) for any $X\in\mathcal{C}$.
			For $X,Y\in \mathcal{C}$, $X<Y$, let $l(X,Y):=l(X)\cap Y$.

			\item[(C3)] For all $X\in \mathcal{C}$, $x\in X$ with the respective neighbourhoods $D(X,x)$ one has:
			\begin{itemize}
				\item $\overline{Y}\cap D(X,x)=\varnothing$ for $Y\in \mathcal{C}$ unless $X\leq Y$;
				\item $\mathcal{C}(X):=\lb l(X,Y)\mid X<Y\in\mathcal{C}\rb$ is a smooth conical stratification of $l(X)$, and its product with the trivial stratification on $D_{T}$ is the induced by $\mathcal{C}$ stratification on $D(X,x)$.
			\end{itemize} 
		\end{enumerate}
	\end{defn}

	The link $l(X,Y)$ does not depend (up to a diffeomorphism) on the choice of $x\in X$ since $X\in \mathcal{C}$ is assumed to be connected.
	Notice that $l(X)$ is a semisphere (i.e. a disk) of dimension $n-\dim X-1$ for $X\neq\rint M$, because $X\subset\partial M$. 
	Therefore, definition~\ref{def:resdec} is inductive, with termination in only finitely many steps. 
	(It is not hard to show that a smooth conical stratification is a Whitney stratification.)

	For subspaces $X,Y\subseteq M$ with $X$ being admissible, the \textit{strict transform} of $Y$ with respect to $p=p_{X}\colon B_{X} M\to M$ is by definition $\widetilde Y:=\overline{p^{-1}(Y\setminus X)}\subseteq B_X M$.
	Notice that a strict transform of a connected admissible smooth submanifold may have corners, and it is neither open nor connected, in general.
	By a slight abuse of the notation we write $\widetilde Y$ also for an iterated blow-up of $Y\subset M$ (if such an object is well defined).

	Let $\mathcal{C}$ be a conical stratification of a smooth manifold $M$ with corners.
	Denote by $B_{\mathcal{C}}M$ the result of blowing up closures of all strata of codimension $\geq 2$ in $M$ in an order that does not decrease dimension.
	Choose any total order $\mathcal{C}=\lb X_{1},\dots, X_{N}\rb$ extending the partial order on $\mathcal{C}$ given by inclusion in the closure.
	Let $M_{i}$ be the iterated blow-up along the first $i$ elements of $\mathcal{C}$.		
	In order to distinguish this total order and the inclusion order, say that $\sigma=(X_{j_1}\subset\cdots\subset X_{j_k})$ is a \textit{flag} in $\mathcal{C}$.
	We record the basic properties of real blow-ups in the following lemma.

	\begin{lm}\label{lm:fullblow}
		\begin{enumerate}[label=(\roman*)]
			\item The subspace $\widetilde{\overline{X_{j}}}$ is admissible in $M_{j-1}$ for all $j=2,\dots,N$; $B_{\mathcal{C}}M$ is a smooth manifold with corners that is independent on the choice of an order on $\mathcal{C}$ up to a diffeomorphism.
			\item Assume that all elements of $\mathcal{C}$ except for the maximal stratum belong to the image of $\partial M$ in $M$.
			Then the blow-up 
			induces a homotopy equivalence of pairs $(B_{\mathcal{C}} M,\partial B_{\mathcal{C}} M)\to (M,\partial M)$.
			\item The boundary hypersurfaces, say, $F_{X}$ of $B_{\mathcal{C}}M$ are connected, and are in bijection with elements of $X\in \mathcal{C}$, $X\neq \rint M$, namely are $\widetilde{E_{X}}$ for $\dim X\leq n-2$, and $\widetilde{\overline{X}}$ for $\dim X=n-1$.
			\item A collection of pairwise distinct boundary hypersurfaces in $B_{\mathcal{C}}M$ has a nonempty intersection iff it forms a flag $X_{j_1}\subset\cdots\subset X_{j_k}$ in $\mathcal{C}$; in particular, a zero dimensional face corresponds to the flag of length $n=\dim M$, and the chart with corners containing it is $\R_{\geq 0}^{n}$.
		\end{enumerate}
	\end{lm}
	\begin{proof}
		The collection $\lb \rint\widetilde{X_{j}},\dots,\rint\widetilde{X_{N}}\rb$ together with all strata of the exceptional divisors for previous blow-ups, is a conical stratification of $M_{j-1}$~\cite[Prop.~9.2]{ga-04}.
		Then $\widetilde{\overline{X_{j}}}$ is admissible by~\cite[Prop.~9.1]{ga-04}, and the blow-up along $\widetilde{X_{j}}$ is well defined.
		The claim on uniqueness is shown in~\cite[Prop.~10.2]{ga-04}.
		This proves~$(i)$.
		
		In order to prove $(ii)$ consider the blow-up $p\colon B_{\widetilde{\overline{X_{j}}}} M_{j-1}\to M_{j-1}$.
		The blow-up map $p$ is proper, and its fibers are contractible (being single points, or links, i.e. closed balls because of having at least one corner).
		Then Smale's theorem~\cite{sm-57} implies that $p$ is a weak equivalence.
		The similar claim holds for the restriction of $p$ to the collar (for the existence of which see~\cite[Thm.~1]{do-61}).
		Since every smooth manifold with corners has the homotopy type of a CW complex~\cite{jo-83}, by Whitehead theorem such a map is a homotopy equivalence.
		Applying the gluing lemma~\cite[7.4.1]{br-68} to the restrictions of $p$ to the collar and to the interior (which agree on the intersection), one obtains that the blow-up map $p$ is a homotopy equivalence.
		This proves $(ii)$.	
		
		By minimality, a stratum $Y$ preceding $X$ in the total order on $\mathcal{C}$ has the strict transform with an empty intersection with $X$; hence, $\widetilde{\overline{X}}=\overline{p_{\overline{Y}}^{-1}(X)}$ is connected as a closure of a connected set.
		Then the strict transforms of elements in $\mathcal{C}$ and of the exceptional divisors (which are connected by the definition) are connected.
		Then the claims $(iii)$, $(iv)$ follow directly from~\cite[Thm.~11.1]{ga-04} (taking the maximal building set).
	\end{proof}
		
	A manifold $M$ with corners is called a \textit{manifold with faces} if any point $x\in M$ lies in exactly $\depth_{M} x$ distinct components of $\partial M$ as defined in~\cite[Def.~2.6]{jo-12}. 
	A manifold with faces $M$ is called \emph{nice} if any face submanifold $F\subset M$ of depth $k$ is an intersection of exactly $k$ boundary hypersurfaces in $M$~\cite{ma-pa-06}.

	\begin{thm}\label{thm:intcorn}
		Let $M$ be any smooth manifold with corners admitting a smooth conical stratification $\mathcal{C}$ such that all elements of $\mathcal{C}$ except for the maximal stratum lie in $\partial M$.
		Then $M$ is homotopy equivalent (compatible with the boundary) to a smooth nice manifold $N:=B_{\mathcal{C}} M$ with faces.
		Suppose that for all $X\subset Y\in \mathcal{C}$ and all $Z\in\mathcal{C}$ (except possibly the maximal stratum) $\overline{l(X,Y)}$ and $\overline{Z}$ are contractible. 
		Then $N$ has only contractible proper (closed) faces.
	\end{thm}	
	\begin{proof}
		The claim on homotopy equivalence is lemma~\ref{lm:fullblow} $(ii)$.
		By lemma~\ref{lm:fullblow} $(iv)$, a codimension $k$ intersection corresponds to a flag of length $k$ non-maximal strata in $\mathcal{C}$, each corresponding to a boundary hypersurface.
		(In particular, the restriction of the boundary map $\iota$ is injective, which corresponds to no self-intersections occurring.)
		This implies the claim on the ``nice'' property, by looking at the chart with corners around points from the relative interior of such a face.
		At a point $x\in B_{\mathcal{C}} M$ of depth $k$ the coordinate subspaces in the local chart of the corner $\partial B_{\mathcal C}M$ correspond by lemma~\ref{lm:fullblow} $(iv)$ to a flag $X_{j_1}\subset\dots\subset X_{j_k}$; its members have pairwise distinct dimensions, hence are pairwise distinct strata, and by lemma~\ref{lm:fullblow} $(iii)$ they index pairwise distinct facets. 
		Hence, $x$ lies in exactly $k$ distinct facets of $B_{\mathcal{C}} M$.
		
		We prove contractibility of faces by the induction on $n=\dim M$.
		For $n=0$ there is nothing to prove.
		Assuming the statement is proved in dimensions $<n$, consider a face $F$ of $B_{\mathcal{C}} M$.
		By lemma~\ref{lm:fullblow}, $F=F_{\sigma}$ corresponds to a flag $\sigma$ in $\mathcal{C}$.
		Let $X=\min \sigma$.		
		The face $E_{X}$ appears after blowing up $X$, and $F_{\sigma}$ is its strict transform in $B_{\mathcal{C}} M$.
		The blow-ups less or incomparable do not affect the lift steps, and other blow-ups modify this face by blowing the respective link factor (by $(C3)$ and \cite[Prop.~9.2]{ga-04}).
		This implies that the restriction of the blow-up map $F_{\sigma}\to \widetilde{\overline{X}}$ has fibers homeomorphic to $F_{\tau}$,
		where $\tau:=( l(X,Y)\mid Y\in\sigma,\ Y>X )$ is a flag in the spherical slice $l(X)$ stratified by $\mathcal{C}(X)$.
		Then by the above, lemma~\ref{lm:fullblow} $(ii)$ and the inductive assumption (applied to $\dim X<n$), the base $\widetilde{\overline{X}}\cong \overline{X}$ is contractible.
		The induction assumption applies to the conical stratification $\mathcal{C}(X)$ of $l(X)$ of dimension $n-\dim X-1$, therefore the fiber is also contractible. 
		This implies the claim on $F_{\sigma}$ by application of Smale's theorem, similar to the above.
	\end{proof}

	\begin{lm}\label{lm:poinc}
		Any contractible compact smooth manifold $M$ with corners of dimension $n\leq 3$ is homeomorphic to a ball.
	\end{lm}
	\begin{proof}
		By smoothening~\cite{do-61}, we may assume that $M$ is a topological manifold.		
		Furthermore, $M$ is contractible, thus orientable.
		By the long exact sequence of $(M,\partial M)$ and the Poincar\'e--Lefschetz duality, 
		\[
			\widetilde{H}_{i}(\partial M;\Z)\cong 
			H_{i+1}(M,\partial M;\Z)\cong
			H^{n-1-i}(M;\Z),\ i\leq n-1.
		\]
		This group vanishes for $i\leq n-2$.
		This implies the claim for $n\leq 2$ by the elementary classification results.
		For $n=3$, $\partial M$ is simply connected.
		Then $M$ is homeomorphic to $D^{3}$ by Perelman, e.g. see~\cite[Cor.~1.2]{be-16}.
	\end{proof}

	\subsection{Conical stratifications of Brieskorn manifolds}

	\begin{figure}
		\centering
		\begin{tikzpicture}
	\draw[thick] (3.40,0.90) -- (4.70,1.65) -- (4.70,3.15)
	-- (3.40,3.90) -- (2.10,3.15) -- (2.10,1.65) -- cycle;
		
	\draw[blue,thick] (4.70,1.65) -- (4.70,3.15); 
	\node[] at (5.05,2.4) {$D_2$};
	\draw[blue,thick] (3.40,3.90) -- (2.10,3.15); 
	\node[] at (2.4,3.65) {$D_3$};
	\draw[blue,thick] (2.10,1.65) -- (3.40,0.90); 
	\node[] at (2.4,1.15) {$D_1$};
	\node at (3.5,2.5) {$H^{+}$};
\end{tikzpicture}
		\caption{The hexagon $H^{+}$ has edges common with the $2$-disks $D_{i}$ around orbifold vertices $v_{i}$, $i=1,2,3$\label{fig:hex}}
	\end{figure}
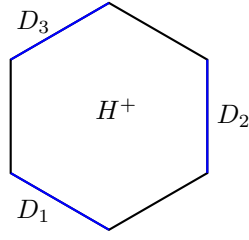			

	Let $\Sigma(p,q,r)$ be any Brieskorn homology sphere with pairwise coprime $p,q,r>1$.
	The manifold $\Sigma(p,q,r)$ has the structure of a locally trivial $S^{1}$-bundle (a Seifert fibration) over the orbifold $2$-sphere $S$ with $3$ cone points $v_{1},v_{2},v_{3}$ of orders $p,q,r$, respectively.
	The base $S$ has a hyperbolic, Euclidean or spherical metric, depending on $p,q,r$ values.
	One identifies $S$ with two copies $T^{+}$, $T^{-}$ of a triangle with geodesic edges and angles $\pi/n$, $n=p,q,r$, glued along boundaries.
	The Seifert invariants~\cite{je-ne-83} of the fibration on $\Sigma(p,q,r)$ are given by pairs $(\alpha_{i},\beta_{i})$ of coprime numbers, $i=1,2,3$, with the respective Euler number $e=\pm 1/pqr$ and parameter $b=-1$, i.e. satisfying the equation
	\[
		e=
		-(b+\sum_{i=1}^{3}\frac{\beta_{i}}{\alpha_{i}}).
	\]	

	\begin{lm}\label{lm:resdecsig}
		There exists a smooth conical stratification by the open cells in a finite regular CW complex on $\Sigma(p,q,r)\times [0,1)$ satisfying all conditions of theorem~\ref{thm:intcorn}.
	\end{lm}
	\begin{proof}
	We work with the boundary of the manifold $M:=\Sigma(p,q,r)\times [0,1)$.
	Cutting small disks $D_{i}$, around (common) vertices $v_{i}$, $i=1,2,3$, of triangles in $S$ we obtain two hexagons $H^{+}\subset T^{+}$, $H^{-}\subset T^{-}$ in the complement, see figure~\ref{fig:hex}.
	We consider each of these $2$-disks $D_{i}$ as a $2$-gon with two vertices and two $1$-dimensional faces at the boundary.
	Such a complex has $6$ vertices, $9$ edges and $5$ (two-dimensional) disks.	
	The Seifert fibration over either of such disks is a solid torus.
	Each of the above solid tori with the invariant, say $(\alpha,\beta)$, is the quotient of a standard solid torus by the free $\Z/\alpha\Z$-action given for the cyclic group generator by
	\[
		(r\exp(2\pi i\phi),\exp(2\pi 	i\psi))\mapsto (r\exp(2\pi i(\phi+1/\alpha)),\exp(2\pi i(\psi+\beta/\alpha))).
	\]
	This is nothing other than the slice $S^{1}\times_{\Z/\alpha\Z} D^{2}$ of the $\Z/\alpha\Z$-action.
	This chart gives rise to the meridian disks $\lb\phi=\const\rb$.
	The boundary of a meridian disk is a circle in the boundary of a two-dimensional torus with the slope $\beta/\alpha$.
	Each torus $T^{2}$ (over the boundary circle of a $2$-disk in $S$) bounds exactly one solid torus, and has a nonempty intersection with $2$ or $4$ other solid tori (preimages of other disks in $S$).
	We subdivide each of the respective solid tori into two equatorial cylinders by meridian circles with generic parameters.
	Each torus $T^{2}$ has the fiber circle (preimage of a vertex in the complex on $S$) and meridian circles or arcs of several slopes.
	Every torus has exactly one family of two meridian circles, and $2$ or $4$ families of arcs lying on two other meridian circles (created by intersection with the remaining incident solid tori).
	Each part of the decomposition on such a torus has no self-intersections.
	Indeed, each torus is cut along its fiber circles; on each resulting annulus there are one or two families of transverse parallel arcs.
	In the first case, the slope of the arc is $0$, and there are two disjoint arcs of such kind, resulting into embedded $4$-gons, so the statement is clear.
	In the second case, since there are at least two arcs in each family (following from $\lb\alpha_{1},\alpha_{2},\alpha_{3}\rb=\lb p,q,r\rb$ and $p,q,r>1$), they cut the annulus into disjoint embedded polygons, proving the claim. 
	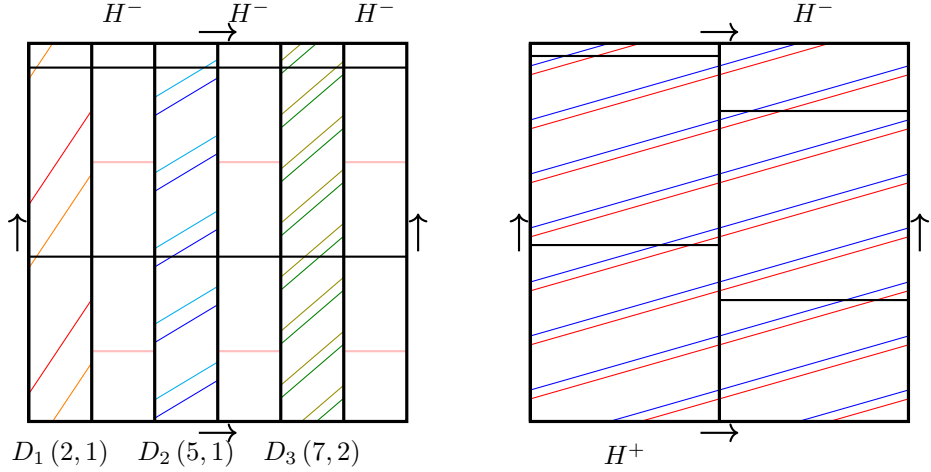
\begin{figure}
		\centering
\begin{tikzpicture}[scale=5]
\useasboundingbox (-0.14,-0.16) rectangle (1.14,1.16);

\draw[red,thin] (0.0000,0.0727) -- (0.1667,0.3227);
\draw[red,thin] (0.0000,0.5727) -- (0.1667,0.8227);
\draw[orange,thin] (0.0000,0.4051) -- (0.1667,0.6551);
\draw[orange,thin] (0.0000,0.9051) -- (0.0633,1.0000);
\draw[orange,thin] (0.0633,0.0000) -- (0.1667,0.1551);
\draw[blue,thin] (0.3333,0.8087) -- (0.5000,0.9087);
\draw[blue,thin] (0.3333,0.0087) -- (0.5000,0.1087);
\draw[blue,thin] (0.3333,0.2087) -- (0.5000,0.3087);
\draw[blue,thin] (0.3333,0.4087) -- (0.5000,0.5087);
\draw[blue,thin] (0.3333,0.6087) -- (0.5000,0.7087);
\draw[cyan,thin] (0.3333,0.2572) -- (0.5000,0.3572);
\draw[cyan,thin] (0.3333,0.4572) -- (0.5000,0.5572);
\draw[cyan,thin] (0.3333,0.6572) -- (0.5000,0.7572);
\draw[cyan,thin] (0.3333,0.8572) -- (0.5000,0.9572);
\draw[cyan,thin] (0.3333,0.0572) -- (0.5000,0.1572);
\draw[green!50!black,thin] (0.6667,0.3474) -- (0.8333,0.4902);
\draw[green!50!black,thin] (0.6667,0.4902) -- (0.8333,0.6331);
\draw[green!50!black,thin] (0.6667,0.6331) -- (0.8333,0.7759);
\draw[green!50!black,thin] (0.6667,0.7759) -- (0.8333,0.9188);
\draw[green!50!black,thin] (0.6667,0.9188) -- (0.7614,1.0000);
\draw[green!50!black,thin] (0.7614,0.0000) -- (0.8333,0.0616);
\draw[green!50!black,thin] (0.6667,0.0616) -- (0.8333,0.2045);
\draw[green!50!black,thin] (0.6667,0.2045) -- (0.8333,0.3474);
\draw[olive,thin] (0.6667,0.2386) -- (0.8333,0.3815);
\draw[olive,thin] (0.6667,0.3815) -- (0.8333,0.5243);
\draw[olive,thin] (0.6667,0.5243) -- (0.8333,0.6672);
\draw[olive,thin] (0.6667,0.6672) -- (0.8333,0.8101);
\draw[olive,thin] (0.6667,0.8101) -- (0.8333,0.9529);
\draw[olive,thin] (0.6667,0.9529) -- (0.7216,1.0000);
\draw[olive,thin] (0.7216,0.0000) -- (0.8333,0.0958);
\draw[olive,thin] (0.6667,0.0958) -- (0.8333,0.2386);

\draw[pink,thick] (0.1667,0.6859) -- (0.3333,0.6859);
\draw[pink,thick] (0.1667,0.1859) -- (0.3333,0.1859);
\draw[pink,thick] (0.5000,0.6859) -- (0.6667,0.6859);
\draw[pink,thick] (0.5000,0.1859) -- (0.6667,0.1859);
\draw[pink,thick] (0.8333,0.6859) -- (1.0000,0.6859);
\draw[pink,thick] (0.8333,0.1859) -- (1.0000,0.1859);

\draw[thick] (0.0000,0.4359) -- (1.0000,0.4359);
\draw[thick] (0.0000,0.9359) -- (1.0000,0.9359);

\foreach \x in {0,0.1667,0.3333,0.5,0.6667,0.8333}{\draw[very thick] (\x,0) -- (\x,1);}
\draw[very thick] (0,0) rectangle (1,1);

\draw[->,thick] (0.45,1.03) -- (0.55,1.03);
\draw[->,thick] (0.45,-0.03) -- (0.55,-0.03);
\draw[->,thick] (1.03,0.45) -- (1.03,0.55);
\draw[->,thick] (-0.03,0.45) -- (-0.03,0.55);

\node at (0.0833,-0.085) {$D_1\,(2,1)$};
\node at (0.2500,1.085) {$H^{-}$};
\node at (0.4167,-0.085) {$D_2\,(5,1)$};
\node at (0.5833,1.085) {$H^{-}$};
\node at (0.7500,-0.085) {$D_3\,(7,2)$};
\node at (0.9167,1.085) {$H^{-}$};

\end{tikzpicture}
\begin{tikzpicture}[scale=5]
\useasboundingbox (-0.14,-0.16) rectangle (1.14,1.16);

\draw[red,thin] (0.0000,0.6312) -- (1.0000,0.9169);
\draw[red,thin] (0.0000,0.7740) -- (0.7909,1.0000);
\draw[red,thin] (0.7909,0.0000) -- (1.0000,0.0597);
\draw[red,thin] (0.0000,0.9169) -- (0.2909,1.0000);
\draw[red,thin] (0.2909,0.0000) -- (1.0000,0.2026);
\draw[red,thin] (0.0000,0.0597) -- (1.0000,0.3455);
\draw[red,thin] (0.0000,0.2026) -- (1.0000,0.4883);
\draw[red,thin] (0.0000,0.3455) -- (1.0000,0.6312);
\draw[red,thin] (0.0000,0.4883) -- (1.0000,0.7740);
\draw[blue,thin] (0.0000,0.6546) -- (1.0000,0.9404);
\draw[blue,thin] (0.0000,0.7975) -- (0.7087,1.0000);
\draw[blue,thin] (0.7087,0.0000) -- (1.0000,0.0832);
\draw[blue,thin] (0.0000,0.9404) -- (0.2087,1.0000);
\draw[blue,thin] (0.2087,0.0000) -- (1.0000,0.2261);
\draw[blue,thin] (0.0000,0.0832) -- (1.0000,0.3689);
\draw[blue,thin] (0.0000,0.2261) -- (1.0000,0.5118);
\draw[blue,thin] (0.0000,0.3689) -- (1.0000,0.6546);
\draw[blue,thin] (0.0000,0.5118) -- (1.0000,0.7975);

\draw[thick] (0.0000,0.9666) -- (0.5000,0.9666);
\draw[thick] (0.0000,0.4666) -- (0.5000,0.4666);
\draw[thick] (0.5000,0.8211) -- (1.0000,0.8211);
\draw[thick] (0.5000,0.3211) -- (1.0000,0.3211);

\draw[very thick] (0,0) -- (0,1);
\draw[very thick] (0.5,0) -- (0.5,1);
\draw[very thick] (0,0) rectangle (1,1);

\draw[->,thick] (0.45,1.03) -- (0.55,1.03);
\draw[->,thick] (0.45,-0.03) -- (0.55,-0.03);
\draw[->,thick] (1.03,0.45) -- (1.03,0.55);
\draw[->,thick] (-0.03,0.45) -- (-0.03,0.55);

\node at (0.25,-0.085) {$H^{+}$};
\node at (0.75,1.085) {$H^{-}$};

\end{tikzpicture}		
		\caption{The $1$-dimensional faces on the tori for $H^{+}$ (left) and $D_{3}$ (around the order $7$ vertex $v_{3}$, right) for $\Sigma(2,5,7)$.
		The vertical circles (including the boundary) correspond to fibers over vertices. 
		The sloped or horizontal segments (not including the boundary) correspond to meridian boundaries or arcs obtained by intersection with other solid tori.
		The meridian arcs have slopes $\beta/2\alpha$ (left, due to only one half of $\partial D_{i}$ intersecting $H^{+}$) and $\beta/\alpha$ (right), $\alpha$ times per fiber circle \label{fig:tor7}}
	\end{figure}	
	All intersections of $1$-dimensional strata are transversal by genericity.
	The vertices of the obtained complex belong to some $T^{2}$, and are of the following two types (which we call $(X)$- or $(T)$-vertices, depending on the shape of the respective $1$-dimensional faces meeting at it):
	\begin{enumerate}
		\item[$(X)$] an intersection of a meridian disk and a fiber circle (both for the same solid torus);
		an intersection of a meridian circle with a meridian arc;	
		\item[$(T)$] an intersection of the fiber circle in one solid torus, and an arc of a meridian disk of another solid torus.
	\end{enumerate}
	The obtained complex is simple, i.e. at every its vertex meet exactly $4$ cylinders.
	Indeed, at a $(T)$-vertex meet two distinct cylinders of $H^{\pm}$, a cylinder from $H^{\mp}$ and a cylinder from the vertex $v_{i}$ neighborhood.
	At an $(X)$-vertex meet a pair of cylinders subdividing one solid torus, and a cylinder from a pair of another disjoint solid tori.
	The condition $(C1)$ from definition~\ref{def:resdec} holds, since there is a unique maximal stratum in $\mathcal{C}$, the closure of any stratum is a union of strata of lower dimensions, and $\partial M$ is a union of strata.
	One can choose the local tubular neighborhoods $D(X,x)$ (shrinking radii near $\overline{X}\setminus X$ if necessary) satisfying $(C2)$ and first bullet of $(C3)$.
	The slice $D_N$ at $X$ is a half-open half-ball of dimension $4-d$, $d=d(X)=\dim X$; its link $l(X)\cong(D_N\smallsetminus0)/\mathbb R_{>0}$ is a closed $(3-d)$-ball, whose interior is the stratum $l(X,\rint M)$ and whose boundary $\partial l(X)\cong S^{2-d}$ carries the remaining strata of $\mathcal C(X)$, described below.
	The strata of $\mathcal C(X)$ are $l(X,\rint M)=\rint l(X)$ together with the links of the cells incident to $X$, which belong to $\partial l(X)$.
	We check the second bullet condition of $(C3)$, i.e. conicality, at points of a stratum $X$ as follows:
	\begin{itemize}
		\item[$d(X)=2$:] the normal bundle of $X\subset M$ has rank $2$, $l(X)$ is an arc, and $\partial l(X)$ consists of two points (links of the incident $3$-cells).
		Then $\mathcal{C}(X)$ consists of $\rint l(X)$ and of these two points, i.e. is conical;
		\item[$d(X)=1$:] along an edge take the tube $X\times D^{3}$ in $M$ in which the incident faces are products of $X$ with a ray from the origin, or a sector.
		Then $l(X)$ is a disc and $\partial l(X)\cong S^1$ is stratified by the links of the incident 2-cells (points) and 3-cells (arcs); together with $\rint l(X)$ this is conical by the previous item;
		\item[$d(X)=0$:] here, $l(X)\cong D^{3}$. 
		By the above, exactly four cylinders meet at the vertex $X$, and the induced stratification of $S^{2}$ is isomorphic to that of $\partial\Delta^3$: four triangles, six arcs, four points, the links of the incident $3$-, $2$- and $1$-cells respectively.
		This is conical by the previous two items. 
		See figure~\ref{fig:stratblowup}.
	\end{itemize}
	Notice that every stratum closure of the decomposition for $\Sigma(p,q,r)$ is contractible: ten $3$-dimensional cylinders, meridian disks, $2$-disks on the $T^{2}$ tori, $1$-dimensional arcs either on such a torus, or those subdividing the fiber circles (at least two in each). 
	We verify contractibility of a link for $X<Y\in\mathcal{C}$ (for $Y=\rint M$, $\overline{l(X,Y)}=\overline{l(X)}$, this was already checked above):
	\begin{itemize}
		\item[$d(X)=2$:] Then $l(X)$ is an arc.
		The only case to consider here is a $3$-cell $Y$, so $l(X,Y)=\lb*\rb$;
		\item[$d(X)=1$:] then $l(X)\cong D^{2}$.
		For an incident $2$-cell $Y$ $l(X,Y)=\lb*\rb$.
		For an incident $3$-cell $Y$ (cylinder) $l(X,Y)$ is an arc;
		\item[$d(X)=0$:] then $l(X)\cong D^{3}$, and $\partial l(X)$ is equipped with the stratification by faces of $\partial \Delta^{3}$ by simplicity, and by the cylinder decomposition.
		Then $l(X,Y)$ is a singleton, an arc or a triangle, depending on $Y$.
	\end{itemize}	
	\end{proof}

	\begin{ex}
		For $\Sigma(2,5,7)$ the Seifert invariants are $(2,1)$, $(5,1)$ and $(7,2)$ at vertices of order $2,5,7$, respectively.
		The corresponding stratifications of tori over boundaries of $H^{+}$ and $D_{3}$ are given in figure~\ref{fig:tor7}.
		Notice that $\pi_{1}(Q_{2})$ is infinite~\cite{mi-75}, since $1/2+1/5+1/7<1$.	
	\end{ex}
		
	\begin{proof}[Proof of theorem~\ref{thm:mainsm}]
	The boundary $\partial W^{4}_{m}$ has a collar with a unique up to a diffeomorphism compatible smooth structure by \cite[Thm.~1]{do-61}.
	By lemma~\ref{lm:resdecsig}, this collar admits a smooth conical stratification.
	Define the conical stratification of $W^{4}_{m}$ on the boundary by the above stratification of the collar, together with the maximal stratum $\rint W^{4}_{m}$.
	Define $Q:=B_{\mathcal{C}} W^{4}_{m}$, which is homotopy equivalent (compatible with the boundary) to $W_{m}^{4}$ by theorem~\ref{thm:intcorn}.
	Sending the facets $F_{X}$ of $B_{\mathcal{C}} W^{4}_{m}$ (see lemma~\ref{lm:fullblow}), $X\in\mathcal{C}$, to $\dim X$ is a proper coloring and defines a characteristic function $\lambda$ on it (compare with~\S\ref{sec:top}).
	Notice that the face poset for $Q$ (the order complex of $\mathcal{C}$) is pure (because closure of any stratum of dimension $>0$ has a stratum of codimension 1), so that maximal flags have equal lengths in $\mathcal{C}$.
	The realization theorem \cite[Thm.~1.1~(3)]{ku-ka-25} produces a smooth torus manifold $N^{8}$ from the triple $(Q,\lambda,0)$.
	Equivariant formality of $N^{8}$ follows from contractibility of the respective closed faces by~\cite[Thm.~2]{ma-pa-06}.
	We claim that $\partial Q$ is a regular CW-complex with the standard cells. 
	By theorem~\ref{thm:intcorn}, every such proper face $F_{X}$ is contractible, and is a transversal intersection of smooth manifolds with corners.
	Hence, $F_{X}$ is a smooth manifold with corners with $\dim F_{X}\leq 3$.
	By lemma~\ref{lm:poinc}, $F_{X}$ is homeomorphic to a closed ball.
	Then the attaching maps of cells are homeomorphisms, and $Q_{3}$ is a regular CW complex, implying $\pi_{1}(Q_{2})=\pi_{1}(Q_{3})$.
	Since $Q\simeq W^{4}_{m}$ is contractible with $\partial Q\simeq \Sigma(p,q,r)$, the proof of $\pi_{1}(N^{8})=1$ carries over verbatim from theorem~\ref{thm:top}.
	\end{proof}

	\section{Complexity $2$ and $3$}\label{sec:thm3proof}

	We briefly recall the necessary definitions for theorem~\ref{thm:mainext}.
	A smooth, closed, connected and simply connected manifold $M=M^{2n}$ equipped with a smooth effective action of a compact torus $T=T^{k}$ is called a \emph{GKM manifold} if $M$ is equivariantly formal, the set $M^{T}$ of $T$-fixed points for $M$ is finite and nonempty, and for every $x\in M^{T}$ the weights of the isotropy $T$-representation at $T_{x} M$ are pairwise linearly independent elements in the weight lattice $\mathfrak{t}^{*}\cong \Z^{k}$ (defined only up to signs).
	Any GKM manifold admits an (unsigned) \emph{GKM graph}, that is, a collection $\Gamma=(V,E,\alpha,\nabla)$, where $(V,E)$ is the $n$-valent graph being the equivariant $1$-skeleton $Q_{1}$ of the action on $M$, i.e. $V=M^{T}$ are its vertices and $E$ is the set of $T$-invariant $2$-spheres in $M$.
	An edge $e\in E$ is pointed, i.e. has the initial and the terminal vertices $i(e)$, $t(e)$, respectively, and is therefore oriented.
	Denote by $\overline{e}\in E$ the edge with the opposite orientation to $e\in E$.
	The collection $\nabla=\lb \nabla_{e}\rb_{e\in E}$ of bijections $\nabla_{e}\colon \str i(e)\to \str t(e)$ satisfying $\nabla_{\overline{e}}:=\nabla_{e}^{-1}$, where $\str v:=\lb e'\in E\colon i(e')=v\rb$, is called a \emph{connection} on $(V,E)$.
	By definition, the function $\alpha\colon E\to \Z^{k}/\lb\pm 1\rb$ satisfies the \textit{congruence relation}
	\[
		\alpha(\nabla_{e}e')=
		\eps_{e,e'}\alpha(e')+c_{e,e'}\alpha(e),\ e,e'\in \str	v
	\]
	for all $v\in V$, some $\eps_{e,e'}\in\lb \pm 1\rb$, $c_{e,e'}\in \Z$, and some choices of the lifts for the three values of $\alpha$ entering the congruence relation along the projection $\Z^{k}\to\Z^{k}/\lb \pm 1\rb$.
	The function $\alpha$ is called an \emph{axial function} on $(V,E,\nabla)$.
	Additionally, it is required that $\alpha$ is \textit{effective}, i.e. the restriction of $\alpha$ to $\str v$ linearly spans the codomain $\Z^{k}$ for every $v\in V$.
	A connected $j$-valent subgraph $F=(V_{F},E_{F})$ of $\Gamma$ is called a \textit{$j$-face} of $\Gamma$ if $\nabla_{e}e'\in E_{F}$ holds for all $e,e'\in E_{F}$ with a common origin.
	The number $c=n-k$ is called the \emph{complexity} of the GKM manifold $M^{2n}$ with the $T^{k}$-action of type $(n,k)$; we say that the respective GKM graph $\Gamma=\Gamma(M)$ has type $(n,k)$.
	The GKM manifold and its GKM graph are called GKM$_{j}$ if the values of $\alpha$ on any pairwise distinct $i\leq j$ edges of $\str v$ are linearly independent.
	One can consider GKM graphs satisfying the above conditions without assuming existence of an underlying torus action on a manifold.
	A GKM graph $\Gamma'$ is an \textit{extension} of a GKM graph $\Gamma$ if the respective graphs and connections coincide, and the respective axial functions satisfy $\alpha=p\circ\alpha'$ for some surjective homomorphism $p\colon \Z^{k'}\to \Z^{k}$.
	Given GKM graphs $\widetilde{\Gamma}$, $\Gamma$, a cover of the graphs $f\colon (\widetilde{V},\widetilde{E})\to (V,E)$ (in the sense of a covering for regular CW complexes) is called a \emph{GKM cover} if the equalities
	\[
		\widetilde{\alpha}(e)=\alpha(f(e)),\ 
		\nabla_{f(e)} f(e')=f(\nabla_{e} e')
	\]
	hold for all edges $e,e'\in \widetilde{E}$ with a common origin.
	A complete obstruction to an extension of a given GKM graph $\Gamma$ is given by the free abelian group $A(\Gamma)$ called the \emph{axial function group}~\cite{ku-19}.
	For the definition of $A(\Gamma)$ see~\cite{go-so-25}.
	Namely, an $(n,k)$-type GKM graph $\Gamma$ has an extension to an $(n,k')$-type GKM graph iff $\rk A(\Gamma)\geq k'$.
	
	\begin{proof}[Proof of theorem~\ref{thm:mainext}]
		We follow the proof of~\cite[Thm.~4]{go-so-25} with several modifications.
		The universal cover $\widetilde{Q_{2}}\to Q_{2}$ induces the GKM cover $\widetilde{\Gamma}\to\Gamma$ with the deck transformation group $H_{1}:=\pi_{1}(Q_{2})$~\cite[Lm.~5.3]{go-so-25}.
		The GKM graph $\widetilde{\Gamma}$ extends to a torus graph by~\cite[Thm.~3]{go-so-25}.
		By~\cite[Lm.~5.4]{go-so-25}, one has $A(\Gamma)=A(\widetilde{\Gamma})^{H_{1}}$ for the $H_{1}$-action on $A(\widetilde{\Gamma})=\Z^{n}$ by invertible over $\Z$ matrices:
		\[
			h\circ a(e):= a(h^{-1}(e)),\ h\in H_{1},\ a\in A(\widetilde{\Gamma}),\ e\in \widetilde{E}.
		\] 
		This representation has an invariant (i.e. trivial representation) split sublattice $\Z^{k}$ which follows from effectivity of the $T$-action, see proof of~\cite[Thm.~4.16]{go-so-25}.
		Therefore, the representation has the block form (where $c=n-k$)
		\[
			H_{1}\ni h\mapsto \begin{pmatrix}
			A(h) & 0\\
			B(h) & \Id
			\end{pmatrix},\ 
			A(h)=(a_{i,j}(h))_{1\leq i,j\leq c},\ 
			B(h)=(a_{i,j}(h))_{1\leq j\leq c<i}.
		\]
		Then, we have a homomorphism $f\colon H_{1}\to GL_{c}(\Z)$, $h\mapsto A(h)$.
		Notice that $\Img f$ is in fact in $\SL_{c}(\Z)$ because $\det\colon H_{1}\to\lb \pm 1\rb$ is trivial (as an abelian representation of a perfect group).
		Therefore, $H_{2}:=H_{1}/\ker f$ is a perfect subgroup of $\SL_{c}(\Z)$.
		Notice that 
		\begin{equation}\label{eq:repid}
			B(gh)=B(h)+B(g)A(h),\ g,h\in H_{1}.
		\end{equation}
		For the first two claims of the theorem it is enough to prove that $H_{2}$ is trivial; the remaining claims then follow by~\cite[Thm.~2]{go-so-26}. 
		Indeed, then $A=\Id$ is a trivial representation and therefore by~\eqref{eq:repid} one has $B(gh)=B(g)+B(h)$ for all $g,h\in H_{1}$.
		Hence, $B=0$ as an abelian representation of the perfect group $H_{1}$.
		Then the $H_{1}$-representation on $A(\widetilde{\Gamma})$ is trivial and $A(\Gamma)=\Z^{n}$ holds.
		Therefore, $\Gamma$ extends to a torus graph by Kuroki's criterion.
		
		Let $c=2$ and consider the image $H_{3}:=p(H_{2})$ for the natural projection $p\colon \SL_{2}(\Z)\to \PSL_{2}(\Z)$.
		Since $\PSL_{2}(\Z)\cong \Z/2\Z*\Z/3\Z$ is a free product, by Kurosh's theorem~\cite[Thm.~5.2]{ma-77} $H_{3}$ is isomorphic to the free product of a free group $F$ of rank $f$ with $a$ copies of $\Z/2\Z$ and $b$ copies of $\Z/3\Z$.
		Since $H_{3}$ is perfect, the cardinals sastisfy $f=a=b=0$, and we have $H_{3}=1$.
		We conclude $H_{2}=1$ from $\ker p=\Z/2\Z$ and $H_{2}$ being perfect.
		
		Let $c=3$ and assume that $H_{1}$ is finite.
		Notice that $H_{2}$ is a finite subgroup of $\SL_{3}(\Z)$, and therefore of $\GL_{3}(\Q)$.
		Its order $d$ satisfies $d\leq 48$ by Minkowski's theorem (e.g. see~\cite[Thm.~1]{se-07}) and therefore is smaller than the least order $60$ of a nontrivial perfect group (namely, of $A_{5}$).
		Hence, $H_{2}=1$.
		This finishes the proof.
	\end{proof}
			
	\noindent {\bf Acknowledgements:} The author gratefully acknowledges funding of the Deutsche Forschungsgemeinschaft
	(DFG, German Research Foundation): Project number 561158824 (Walter Benjamin Fellowship). 
	The author is grateful to O.~Goertsches and M.~Wiemeler for fruitful discussions on the subject of the note.

\end{document}